\documentclass[11pt,english]{smfart}

\usepackage[T1]{fontenc}
\usepackage{dsfont,fourier,baskervald}

\usepackage[text={6.5in, 8.5in},centering]{geometry}
\usepackage[all]{xy}
\usepackage[dvipsnames]{xcolor}   
\usepackage{xparse}
\usepackage{xr-hyper}
\usepackage[linktocpage=true,colorlinks=true,hyperindex,citecolor=teal,linkcolor=blue]{hyperref}

\newtheorem{theorem}{Theorem}[section]

\theoremstyle{definition}

\newtheorem{remark}[theorem]{Remark}

\numberwithin{equation}{section}

\renewcommand{\labelenumi}{\upshape(\arabic{enumi})}

\begin{document}

\title[Lawvere's Frobenius reciprocity]{Lawvere's Frobenius reciprocity, the modular connections of Grandis and Dilworth's abstract principal ideals}

\author{Amartya Goswami}
\address{[1] Department of Mathematics and Applied Mathematics, University of Johannesburg, Cnr Kingsway and University Road, 
	Johannesburg, 2006, Gauteng, South Africa. [2] National Institute for Theoretical and Computational Sciences (NITheCS),  South Africa}
\email{agoswami@uj.ac.za}

\author{Zurab Janelidze}
\address{[1] Mathematics Division, Department of Mathematical Sciences, Stellenbosch University, Merriman Street, Stellenbosch, 7602, Western Cape, South Africa. [2] National Institute for Theoretical and Computational Sciences (NITheCS),  South Africa}
\email{zurab@sun.ac.za}

\author{Graham Manuell}

\address{[1] Mathematics Division, Department of Mathematical Sciences, Stellenbosch University, Merriman Street, Stellenbosch, 7602, Western Cape, South Africa. [2] National Institute for Theoretical and Computational Sciences (NITheCS),  South Africa}
\email{graham@manuell.me}

\begin{abstract}
The purpose of this short note is to fill a gap in the literature: Frobenius reciprocity in the theory of doctrines is closely related to modular connections in projective homological algebra and the notion of a principal element in abstract commutative ideal theory. These concepts are based on particular properties of Galois connections which play an important role also in the abstract study of group-like structures from the perspective of categorical/universal algebra; such role stems from a classical and basic result in group theory: the lattice isomorphism theorem.
\end{abstract}



\maketitle

\section{Introduction}
This paper concerns three fundamental concepts: \emph{Frobenius reciprocity} from categorical logic, \emph{modular connections} from non-abelian homological algebra and \emph{principal elements} from abstract ideal theory.
\begin{itemize}
  \item Frobenius reciprocity is a condition relating images, preimages and intersection that makes existential quantification interact well with logical conjunction \cite{Law70}.
  \item Modular connections arise in an attempt to capture the homological behaviour of modules in lattice theoretical terms (see \cite{Gra12,Gra13} and the references there).
  \item Principal elements in a multiplicative lattice approximate the behaviour of principal ideals of a commutative ring (with identity) in the integral commutative quantale of all ideals \cite{Dil62}.
\end{itemize}
We will explain how each of these concepts are instances of the same idea, and also point out links with closely related concepts in residuation theory \cite{BJ72} and various interlinked abstract approaches in the study of group-like structures \cite{MB99,GJ19,BB04,Urs73}.

\section{Connections}

Given posets $P$ and $Q$, by a \emph{connection} we mean a relation $R\colon P\to Q$ such that 
$$a\leqslant b\; R\; c\leqslant d\quad\Rightarrow\quad a\;R\; d,$$
for all $a$, $b\in P$ and $c$, $d\in Q$. Note that a relation $R\colon P\to Q$ is a connection if and only if the opposite relation $R^\mathsf{op}$ is connection between dual posets, $R^\mathsf{op}\colon Q^\mathsf{op}\to P^\mathsf{op}$.

\begin{remark}
  What we call connections in this paper are often referred to as ``weakening relations'' in, especially, computer science literature (see e.g., \cite{JS23}, and the references there).
  When posets are viewed as categories, connections are $\mathbf{2}$-valued profunctors.
\end{remark}

We say that a connection $R$ is a \emph{left adjoint connection} when there is a map $f_R\colon P\to Q$ (called a \emph{left adjoint}) such that 
$$f_R(x)\leqslant y\quad\Leftrightarrow\quad x\;R\;y,$$ for all $x\in P$ and $y\in Q$. Symmetrically, we say that a connection is a \emph{right adjoint connection} when there is a map $f^R\colon Q\to P$ (called a \emph{right adjoint}) such that 
$$x\;R\;y\quad\Leftrightarrow\quad x\leqslant f^R(y).$$
Note that the concepts of left and right adjoint maps are dual to each other: left/right adjoint for $R$ is the same as a right/left adjoint for $R^\mathsf{op}$. 

Piecing the conditions defining left/right adjoint maps together, we get:
$$f_R(x)\leqslant y\quad\Leftrightarrow\quad x\;R\;y\quad\Leftrightarrow\quad x\leqslant f^R(y).$$ 
We call a connection $R$ admitting both a left adjoint and a right adjoint, an \emph{adjoint connection}. Of course, every adjoint connection $R$ gives rise to a Galois connection $(f_R,f^R)$, since the middle expression in the pair of equivalences above can be dropped:
$$f_R(x)\leqslant y\quad\Leftrightarrow\quad x\leqslant f^R(y).$$ 
But it can just as well be inserted back: define $R$ by these equivalent formulas. It is obvious that the relation reconstructed from $(f_R,f^R)$ will be the $R$ we started with. The following result then gives that there is a one-to-one correspondence between adjoint connections and Galois connections. The left adjoints in adjoint connections are known as ``residuated mappings'' (see e.g., \cite{BJ72}).

\begin{theorem}\label{ThmB}
Every connection $R\colon P\to Q$ has at most one left adjoint map $f_R$. Moreover, if $f_R$ exists, it is monotone. In fact, there is a one-to-one correspondence between left adjoint connections $R\colon P\to Q$ and monotone maps $f\colon P\to Q$.
\end{theorem}

\begin{proof}
Suppose $f_R$ and $f'_R$ are both left adjoint maps for the same connection $R$. Then 
$$f_R(x)\leqslant f'_R(x)\quad\Leftrightarrow\quad x\;R\;f'_R(x)\quad\Leftrightarrow\quad f'_R(x)\leqslant f'_R(x),$$
for every $x\in P$. We have a similar chain of equivalences with the roles of $f'_R$ and $f_R$ swapped around, and so $f'_R=f_R$. Next, we prove that $f_R$ is monotone. Suppose $x\leqslant y$ in $P$. We have $y\;R\;f_R(y)$ since $f_R(y)\leqslant f_R(y)$. Then $x\;R\;f_R(y)$ and so $f_R(x)\leqslant f_R(y)$. 

Clearly, the relation $R$ can be recovered from $f_R$. It is also obvious that a monotone $f$ is the same as $f_R$ for $R$ defined by
$$x\;R\;y\quad\Leftrightarrow\quad f(x)\leqslant y.$$
So all it remains to show is that when $f$ is monotone, $R$ is a connection. Suppose $a\leqslant b\;R\;c\leqslant d$. Then $$f(a)\leqslant f(b)\leqslant c\leqslant d $$
and so $a\;R\;d$.
\end{proof}

Dually to the previous theorem, we have:

\begin{theorem}\label{ThmC}
Every connection $R\colon P\to Q$ has at most one right adjoint map $f^R$. Moreover, if $f^R$ exists, it is monotone. In fact, there is a one-to-one correspondence between right adjoint connections $R\colon P\to Q$ and monotone maps $f\colon Q\to P$.
\end{theorem}

\begin{remark}
  The concept of a Galois connection between posets, in its full generality, goes back at least to \cite{Ore44}. In categorical terms, they are adjunctions between posets and the above theorems are specializations of well-known facts about representable profunctors.
\end{remark}

\section{Modular connections}

When $P$ and $Q$ are bounded modular lattices, a connection $R\colon P\to Q$ is \emph{modular} in the sense of M.~Grandis (see \cite{Gra13} and the references there) when it is an adjoint connection and the two dual laws below hold, where $\bot$ denotes the bottom element of $Q$ and $\top$ denotes the top element of $P$.
\begin{itemize}
\item[(RM0)] $f^R(f_R(x))=x\vee f^R(\bot)$, 
\item[(LM0)] $f_R(f^R(y))=y\wedge f_R(\top)$. 
\end{itemize}
In the setting of modular lattices, the laws (RM0) and (LM0) are equivalent to (RF0) and (LF0) (which we state in Section~\ref{sec:principal}).

The laws above play a significant role in projective homological algebra \cite{Gra12,Gra13}. They are also important beyond the setting of modular lattices. As shown in \cite{Jan14}, they characterize  protomodularity \cite{B91} of regular categories (when applied to direct-inverse image adjunctions between subobject posets), which play an important role both in categorical algebra and in universal algebra, where they are closely related to BIT varieties in the sense of A.~Ursini \cite{Urs73} (see \cite{BB04} and the references there). The work \cite{Jan14} led to ``noetherian forms'' \cite{GJ19}, a self-dual framework for establishing homomorphism theorems for group-like structures. The laws above are among the axioms for this framework. These developments are closely related to the fact that in group theory, the laws (RM0) and (LM0) describe the Lattice Isomorphism Theorem, which, as shown in \cite{MB99}, can be used to derive other isomorphism theorems.

The significance of modular connections in the broader theory of Galois connections is recognized in \cite{BJ72}, where they are called ``weakly regular'' mappings. See Example~13.2 there for how modular connections relate to the property of modularity of a lattice (see also \cite{Man19}).

The following result is a direct consequence of Proposition~1 in \cite{Jan14}, hence we omit the proof. The equivalence of (LM0) and (LM1) in this theorem can also be found in \cite{NZ14}; in the slightly more restricted context of bounded lattices, this equivalence has been proved already in \cite{BJ72}.

\begin{theorem} For any adjoint connection $R\colon P\to Q$ between posets $P$ and $Q$, the following conditions are equivalent:
\begin{itemize}
\item[(LM1)] The direct image of $f_R$ is down-closed, i.e., whenever $c\leqslant f_R(b)$, we have $c=f_R(a)$ for some $a\in P$.

\item[(LM2)] Whenever $c\leqslant d$ in $Q$, the element $f_R(f^R(c))$ is the meet of $c$ and $f_R(f^R(d))$, i.e., $f_R(f^R(c))=c\wedge f_R(f^R(d))$.

\item[(LM3)] Whenever some meet $c\wedge d$ exists in $Q$, we have: $f_R(f^R(c\wedge d))=c\wedge f_R(f^R(d))$.
\end{itemize}
Whenever $P$ has a top element, these conditions are further equivalent to (LM0).
\end{theorem}

Dually, we have:

\begin{theorem} For any adjoint connection $R\colon P\to Q$ between posets $P$ and $Q$, the following conditions are equivalent:
\begin{itemize}
\item[(RM1)] The direct image of $f^R$ is up-closed, i.e., whenever $b\geqslant f^R(c)$, we have $b=f^R(d)$ for some $a\in P$.

\item[(RM2)] Whenever $b\geqslant a$ in $P$, the element $f^R(f_R(b))$ is the join of $b$ and $f^R(f_R(a))$, i.e., $f^R(f_R(b))=b\vee f^R(f_R(a))$.

\item[(RM3)] Whenever some join $a\vee b$ exists in $P$, we have: $f^R(f_R(a\vee b))=a\vee f^R(f_R(b))$.
\end{itemize}
Whenever $Q$ has a bottom element, these conditions are further equivalent to (RM0).
\end{theorem}

\section{Principal elements and Frobenius reciprocity}\label{sec:principal}

%In \cite{Dil62}, R.~P.~Dilworth introduced the following concept. 

Consider now a commutative quantale, i.e., a complete lattice $L$ equipped with an associative and commutative binary operation 
$$L\times L\rightarrow L,\quad (a,b)\mapsto a\cdot b = ab$$ such that $\cdot$ preserves joins in each argument. By well-known properties of Galois connections, for each fixed $e\in L$, the map
$$b\mapsto be$$
is a left adjoint in an adjoint connection, where the values of the right adjoint are written as follows:
$$a\mapsto a:e.$$
In the context of such lattice, in abstract commutative ideal theory \cite{Dil62} (see also \cite{AJ01} and the references there), an element $e$ is said to be \emph{principal} when the following dual laws hold:
(i) $c\wedge de=((c:e)\wedge d)e$, (ii) $a\vee (b:e)=(ae\vee b):e$. These properties are identical to the laws (LF0) and (RF0), respectively, of a general adjoint connection:
\begin{itemize}
\item[(RF0)] $a\vee f^R(c)=f^R(f_R(a)\vee c)$.
\item[(LF0)] $c\wedge f_R(b)=f_R(f^R(c)\wedge b)$.
\end{itemize}
The law (LF0) was called \emph{Frobenius reciprocity} in \cite{Law70}. It plays a significant role in the theory of doctrines -- a functorial approach to mathematical logic.
In a slightly more abstract form, it is one of the axioms (``modular law'') for the concept of an allegory \cite{FS90}, which is an abstraction of the category of relations in a regular category.
The laws (RF0) and (LF0) are easily seen to be stable under composition.

Note that we can get (RM0) from (RF0) by setting $c=\bot$. Dually, by setting $b=\top$ in (LF0) we obtain (LM0).

The laws (RM0) and (LM0) also feature in abstract commutative ideal theory. They define ``weak principal elements''. Lemma~1(a) from \cite{AJ96} is a special case of a general property of adjoint connections: the equivalence of (LM0) and (LM1). Lemma~1(d) from \cite{AJ96} generalizes similarly to arbitrary adjoint connections. For semilattices, the equivalence of (RM0) and (RM5) below has been noted in \cite{NZ14}, whose very goal is to abstract the theory of principal elements to mappings (connections satisfying (LF0) and (RF0) are called ``principal mappings'' in \cite{NZ14}). For complete lattices, equivalence of (RM0) and (RM4) is noted in \cite{Man19}.

\begin{theorem}
Let $R$ be a connection from a poset $P$ to a poset $Q$, where $P$ has binary joins and $Q$ has a bottom element. Then (RM0) is equivalent to each of the the following laws:
\begin{itemize}
\item[(RM4)] $f_R(a)=f_R(b)$ implies $a\vee f^R(\bot)=b\vee f^R(\bot)$.

\item[(RM5)] $f_R(a)\leqslant f_R(b)$ implies $a\leqslant b\vee f^R(\bot)$.
\end{itemize}
\end{theorem}

\begin{proof}
(RM0) clearly implies (RM4). Suppose (RM4) holds and $f_R(a)\leqslant f_R(b)$. Then $$f_R(b)=f_R(a)\vee f_R(b)=f_R(a\vee b)$$ (since in a Galois connection, the left adjoint map preserves joins). By (RM4), we then have 
$$b\vee f^R(\bot)=a\vee b\vee f^R(\bot),$$
i.e., $a\leqslant b\vee f^R(\bot)$. Thus, (RM4) implies (RM5).

Next, we show that (RM5) implies (RM0). Suppose (RM5) holds. Since $f_R(f^R(f_R(b)))=f_R(b)$ (by a general property of Galois connections), we have
$$f^R(f_R(b))\leqslant b\vee f^R(\bot).$$ The reverse inequality holds for any adjoint connection. This proves (RM0).
\end{proof}

Dually, we have:

\begin{theorem}
Let $R$ be a connection from a poset $P$ to a poset $Q$, where $P$ has a top element and $Q$ has binary meets. Then (LM0) is equivalent to each of the following laws:
\begin{itemize}
\item[(LM4)] $f^R(c)=f^R(d)$ implies $c\wedge f_R(\top)=d\wedge f_R(\top)$.

\item[(LM5)] $f^R(d)\geqslant f^R(c)$ implies $d\geqslant c\wedge f_R(\top)$.
\end{itemize}
\end{theorem}

The following is a generalization of Lemma 1(b,e) from \cite{AJ96} to arbitrary adjoint connections. Connections between lattices that satisfy (LF0) and (RF0) are called ``strongly range-closed'' mappings in \cite{BJ72}, which contains the result about the equivalence of (LF1) and (LF0) for bounded lattices, and for lattices: in the form of an exercise with a hint; the proof presented below makes use of this hint too (this can also be found in \cite{NZ14}). Below we use the notation $a^\downarrow=\{b\mid b\leqslant a\}$ and $a^\uparrow=\{b\mid b\geqslant a\}$ for an element $a$ in a given poset.

\begin{theorem}\label{ThmA}
Let $R$ be a left adjoint connection between posets, $P\to Q$. The following conditions are equivalent: 
\begin{itemize}
\item[(LF1)] If $c\leqslant f_R(b)$ for $c\in Q$ and $b\in P$, then $c=f_R(a)$ for some $a\in P$ such that $a\leqslant b$.

\item[(LF2)] (LM1) holds for the restriction $a^\downarrow\to f_R(a)^\downarrow$ of $f_R$, for each $a\in P$.
\end{itemize}
When $R$ is an adjoint connection and $P$ and $Q$ have binary meets, the conditions above are further equivalent to (LF0). 
\end{theorem}

\begin{proof}
The equivalence of (LF1) and (LF2) is trivial. Now assume $R$ is an adjoint connection and that $P$ and $Q$ have binary meets. Suppose (LF1) holds. The inequality $f_R(f^R(c)\wedge b)\leqslant c\wedge f_R(b)$ holds for any adjoint connection. Since $c\wedge f_R(b)\leqslant f_R(b)$, we get that $c\wedge f_R(b)=f_R(a)$ for some $a\leqslant b$. Then $f_R(a)\leqslant c$ and so $a\leqslant f^R(c)$. This gives
$$c\wedge f_R(b)=f_R(a)\leqslant f_R(f^R(c)\wedge b),$$ as desired.
Conversely, the implication (LF0) $\Rightarrow$ (LF1) is immediate.
\end{proof}

Dually, we have:

\begin{theorem}
Let $R$ be a right adjoint connection between posets, $P\to Q$. The following conditions are equivalent: 
\begin{itemize}
\item[(LF1)] If $b\geqslant f^R(c)$ for $b\in P$ and $c\in Q$, then $b=f^R(d)$ for some $d\in Q$ such that $d\geqslant c$.

\item[(LF2)] (RM1) holds for the restriction $c^\uparrow\to f^R(c)^\uparrow$ of $f^R$, for each $c\in P$. 
\end{itemize}
When $R$ is an adjoint connection and $P$ and $Q$ have binary joins, the conditions above are further equivalent to (RF0). 
\end{theorem}

Characterizations of adjoint connections between complete lattices satisfying the conditions studied in this paper, motivated by viewing connections as morphisms between complete lattices, can be found \cite{Man19}. This approach relates closely to the work of Grandis (see \cite{Gra12,Gra13} and the references there), where adjoint connections between various lattices are studied as morphisms in suitable categories. See also \cite{NZ14,BJ72} (cited in \cite{Man19}) for related results.

It is shown in \cite{NZ14} that for an adjoint connection $R\colon P\to Q$ between bounded lattices, if $P$ is modular, then (LM0)$\Leftrightarrow$(LF0) and if $Q$ is modular, then (RM0)$\Leftrightarrow$(RF0). This refines the observation made in \cite{Gra12} that: when $P$ and $Q$ are both modular lattices, the conjunction of (LM0) and (RM0) is equivalent to the conjunction of (LF0) and (RF0). It follows from this observation that modular connections in Grandis exact categories \cite{Gra12} satisfy Frobenius reciprocity and its dual. The same is not true for noetherian forms, as noted in \cite{DGJ24}, where the link between \cite{Gra12} and \cite{Law70} was first pointed out.

\begin{remark}
 As noted in \cite{Jan14}, in categorical terms, (LM0) states that the counit of the adjunction is a cartesian natural transformation. As also noted there, (RM0) is equivalent to the monad corresponding to the adjunction being nullary in the sense of \cite{CJ95}. As far as we know, apart from a minor remark in \cite{Jan14}, the extension of the characterizations of (LM0) and (LF0) contained in this paper to adjunctions between general categories has not been pursued.
\end{remark}

\section{Conclusion}

The present paper brings together two streams of literature that explore conditions on connections considered in this paper. These two streams are represented by \cite{DGJ24} and \cite{Man19}, and appropriate references there. As these conditions, in one or another form, have arisen independently in a number of unrelated works \cite{Dil62,Law70,Urs73,MB99,BB04,Gra12}, it is quite likely that there are also other streams of literature that we have missed and that rediscover and/or exploit the same conditions. Compiling a complete library of all such sources could be worthwhile.


\begin{thebibliography}{99}
\bibitem{AJ96} Anderson DD, Johnson EW (1996) Dilworth's principal elements. Algebra Universalis 36:392--404

\bibitem{AJ01} Anderson DD, Johnson EW (2001) Abstract Ideal Theory from Krull to the Present. In:
Ideal Theoretic Methods in Commutative Algebra. Marcel Dekker, p 27--47

\bibitem{BJ72} Blyth TS, Janowitz MF (1972) Residuation Theory. In: International Series of Monographs in Pure and Applied Mathematics. 2, Pergamon Press, Oxford

\bibitem{BB04}
Borceux F, Bourn D (2004) Mal’cev, Protomodular, Homological and Semi-Abelian Categories, Math Appl 566, Kluwer

\bibitem{B91} Bourn D (1991) Normalization equivalence, kernel equivalence, and affine categories, Springer, p 43--62

\bibitem{CJ95}
Carboni A, Janelidze G (1995) Modularity and descent. J Pure Appl Algebra 99:255--265

\bibitem{DGJ24} Dayaram K, Goswami A,  Janelidze Z (2024) Isbell's subfactor projections in a noetherian form. Advanced Studies: Euro-Tbilisi Mathematical Journal 17:63--78


\bibitem{Dil62} Dilworth RP (1962) Abstract commutative ideal theory. Pacific J Math 12:481--
498

\bibitem{FS90} Freyd P, Scedrov, A (1990) Categories, Allegories. Mathematical Library, 39, North-Holland

\bibitem{GJ19}
Goswami A,  Janelidze Z (2019) 
Duality in non-abelian Algebra IV. Duality for groups and a universal isomorphism theorem. Adv Math 349:781--812.

\bibitem{Gra12} Grandis M (2012) Homological Algebra: The interplay of homology with distributive lattices and orthodox semigroups. World Scientific Publishing Co., Singapore

\bibitem{Gra13}
Grandis M (2013) Homological algebra in strongly non-abelian settings. World Scientific Publishing Co., Singapore

\bibitem{Jan14} Janelidze Z (2014) On the Form of Subobjects in Semi-Abelian and Regular Protomodular Categories. Appl Categ Structures 22:755-766

\bibitem{JS23} Jipsen P,  Šemrl J  (2023) Representable and Diagonally Representable Weakening Relation Algebra. In: Glück, R., Santocanale, L., Winter, M. (eds) Relational and Algebraic Methods in Computer Science. RAMiCS. Lecture Notes in Computer Science, vol 13896. Springer, Cham

\bibitem{Law70} Lawvere FW, Equality in hyperdoctrines and comprehension schema as an adjoint functor (1970)  Proceedings of the AMS Symposium on Pure Mathematics XVII:1--14 

\bibitem{MB99} Mac Lane S, Birkhoff GD (1999) Algebra, third edition, AMS Chelsea Publishing

\bibitem{Man19} Manuell G (2019) Quantalic spectra of semirings, PhD Thesis, University of Edinburgh

\bibitem{NZ14} Nai YT, Zhao D (2014) Principal Mappings between Posets. International Journal of Mathematics and Mathematical Sciences, Article ID 754019

\bibitem{Ore44} Ore O (1944) Galois connexions.  Trans Amer Math Soc. 55:493-513

\bibitem{Urs73} Ursini A (1973) Osservazioni sulle variet a BIT. Boll Un Mat Ital  7:205-211

%\bibitem{War37}
%M. Ward, Residuation in structures over which a multiplication is defined, \textit{Duke Math. J.} \textbf{3}, 1937, 627--636.

%\bibitem{WD39}  M.~Ward, and R.~P.~Dilworth, Residuated lattices, \textit{Trans. Amer. Math. Soc.} \textbf{45}, 1939, 335--354.
\end{thebibliography}
\end{document}